\documentclass{amsart}
\usepackage{amsmath, amssymb, amsthm, amsfonts}
\usepackage{hyperref}
\usepackage{cleveref}
\usepackage[a4paper,left=20mm,right=20mm,top=30mm,bottom=25mm]{geometry}
\usepackage{kotex}
\usepackage{blkarray}
\usepackage{enumerate}
\usepackage{xcolor}
\usepackage{stmaryrd} 
\usepackage{graphicx}
\usepackage[all]{xy} 
\usepackage{mathrsfs}
\usepackage{eufrak}
\usepackage{comment}
\usepackage{mathtools}

\theoremstyle{plain}
\newtheorem{thm}{Theorem}[section]
\newtheorem{lem}[thm]{Lemma}
\newtheorem{prop}[thm]{Proposition}

\theoremstyle{definition}

\theoremstyle{remark}
\newtheorem{rem}[thm]{Remark}

\newcommand{\bl}{\textrm{bl}}

\newcommand{\pr}{\textrm{pr}}

\newcommand{\lra}{\longrightarrow}

\newcommand{\cO}{\mathcal{O}}
\newcommand{\cI}{\mathcal{I}}

\newcommand{\res}{\mathrm{res}}

\title[Hilbert polynomials of secant varieties]{Hilbert polynomials of the first and second secant varieties}

\author{Doyoung Choi}
\address{School of Mathematics, Korea Institute for Advanced Study, 85 Hoegiro, Dongdaemun-gu, Seoul 02455, Republic of Korea}
\email{cdy4019@kias.re.kr}

\author{Jinhyung Park}
\address{Department of Mathematical Sciences, KAIST, 291 Daehak-ro, Yuseong-gu, Daejeon 34141, Republic of Korea}
\email{parkjh13@kaist.ac.kr}

\date{\today}

\thanks{D.C. was supported by a KIAS individual Grant (MG105101) at Korea Institute for Advanced Study. J.P. was partially supported by the National Research Foundation (NRF) funded by the Korea government (MSIT) (RS-2026-25478877).}

\begin{document}

\begin{abstract}
In this paper, we give a description of the cohomology groups of the symmetric powers of the tautological bundle associated with a sufficiently positive line bundle on the Hilbert scheme of 2 or 3 points on a smooth projective complex variety. The dimensions of these cohomology groups are expressed in terms of the Hilbert polynomials of the first and second secant varieties of the embedding given by the sufficiently positive line bundle. We also compute these Hilbert polynomials completely.
\end{abstract}

\maketitle
\section{Introduction}

\noindent Let $X$ be a smooth projective complex variety of dimension $n$, and $L$ be a line bundle on $X$. Denote by $X^{[k+1]}$ the Hilbert scheme  of $k+1$ points on $X$, and consider the tautological bundle  $E_{k+1,L}$ on $X^{[k+1]}$ associated to $L$ (see Section \ref{sec:prelim} for the definitions and notations). The cohomology groups $H^i(X^{[k+1]}, S^\ell E_{k+1,L})$ of symmetric powers of the tautological bundle $E_{k+1,L}$ has attracted considerable attention (see e.g., \cite{Danila1, Danila2, ENP2, Krug, Scala}). This has been studied mainly for the case where $n$ and $\ell$ are small and $k$ is arbitrary. In this paper, we deal with the case where $n$ and $\ell$ are arbitrary and $k=1$ or $2$ under the assumption that $L$ is sufficiently positive (see Remark \ref{rem:suffposcond} for the discussion on the required positivity condition for $L$). Our result may provide a complement to the existing literature.

\begin{thm}\label{thm:main1}
Assume that $L$ is sufficiently positive. For an integer $\ell \geq 1$, we have the following:
\begin{enumerate}
    \item $H^i(X^{[2]}, S^\ell E_{2,L}) = 
    \begin{cases} 
    H^0(\Sigma_1, \mathcal{O}_{\Sigma_1}(\ell)) & \text{for $i=0$} \\ 
    H^i(X, \mathcal{O}_X) \otimes H^0(X, \mathcal{O}_X(\ell)) & \text{for $1 \leq i \leq n$} \\
    0 & \text{for $i \geq n+1$}.
    \end{cases}$
    \item $H^i(X^{[3]}, S^\ell E_{3, L}) = 
    \begin{cases}
    H^0(\Sigma_2, \mathcal{O}_{\Sigma_2}(\ell)) & \text{for $i=0$} \\
    \begin{aligned} 
    & \big( H^i(X, \mathcal{O}_X) \otimes H^0(\Sigma_1, \mathcal{O}_{\Sigma_1}(\ell)) \big) \\[-7pt] 
    & \oplus \big(H^i(X^{[2]}, \mathcal{O}_{X^{[2]}})/H^i(X, \mathcal{O}_X) \otimes H^0(X, \mathcal{O}_X(\ell))\big) 
    \end{aligned} & \text{for $1 \leq i \leq n$}\\
    H^i(X^{[2]}, \mathcal{O}_{X^{[2]}}) \otimes H^0(X, \mathcal{O}_X(\ell)) & \text{for $n+1 \leq i \leq 2n$}\\
    0 & \text{for $i \geq 2n+1$}.
    \end{cases}$
\end{enumerate}
\end{thm}

This theorem is a higher dimensional generalization of \cite[Theorem 4.2 (2)]{ENP2}, which deals with the case where $X$ is a curve. Since the cohomology groups $H^i(X^{[2]}, \mathcal{O}_{X^{[2]}})$ can be explicitly determined (Lemma \ref{lem:H^i(O_X^[2])}), the problem is reduced to computing the Hilbert function $H_{\Sigma_k}(\ell)=h^0(\Sigma_k, \mathcal{O}_{\Sigma_k}(\ell))$ of the $k$-th secant variety $\Sigma_k$ of $X \subseteq \mathbb{P} H^0(X, L) = \mathbb{P}^r$ for $k=1$ and $2$, which is defined as
$$
\Sigma_k= \Sigma_k(X, L) := \overline{\bigcup_{x_1, \ldots, x_{k+1} \in X} \langle x_1, \ldots, x_{k+1} \rangle} \subseteq \mathbb{P}^r.
$$
Here as $L$ is sufficiently positive, $L$ is very ample giving an embedding of $X$ into $\mathbb{P}^r$. There has been a great deal of work on secant varieties in various flavors (see e.g., \cite{AP25, Choi.Lacini.Park.Sheridan.25, Choi, Chou.Song.18, ENP1, ENP2, Olano.Raychaudhury.Song, Ullery}). Recently, many people are interested in the connection to the tensor decomposition problem and computational complexity theory (see \cite{Landsberg.Tensors, Landsberg.Complexity}).
In the main cases of this paper, namely, $k=1$ or $2$, it is known that $H^i(\Sigma_k, \mathcal{O}_{\Sigma_k}(\ell)) = 0$ for $i>0$ and $\ell > 0$ (see Theorem \ref{thm:secantvariety}). Thus $h^0(\Sigma_k, \mathcal{O}_{\Sigma_k}(\ell))= \chi(\mathcal{O}_{\Sigma_k}(\ell))$ for $\ell \geq 1$. In other words, the Hilbert function $H_{\Sigma_k}(\ell)$ coincides with the Hilbert polynomial $P_{\Sigma_k}(\ell):=\chi(\mathcal{O}_{\Sigma_k}(\ell))$. In this paper, we express $P_{\Sigma_k}(\ell)$ in terms of $\Gamma:=\sum_{i=1}^n (-1)^n h^p(X, \mathcal{O}_X)$ and the following functions
$$
l(a):=\chi(L^a),~s_1(a,b):=\chi(S^a \Omega_X \otimes L^b),~s_2(a,b,c):=\chi(S^a \Omega_X \otimes S^b \Omega_X \otimes L^c)
$$
depending only on the cotangent bundle $\Omega_X$ and the embedding line bundle $L$. 

\begin{thm}\label{thm:main2}
Assume that $L$ is sufficiently positive. For an integer $\ell \in \mathbb{Z}$, we have the following: 
\begin{enumerate}
    \item $\displaystyle P_{\Sigma_1}(\ell) = \sum_{m=0}^{2n+1} \left( l(2m+1) + \sum_{i=1}^{m} \Big(l(2m+1-i)l(i) - \sum_{j=0}^{2i-1} s_1(j,2m+1) \Big) \right) \prod_{\substack{t=0 \\ t \neq m}}^{2n+1} \frac{\ell - 2t-1}{2m - 2t}$.
    \item $\displaystyle P_{\Sigma_2}(\ell) = \sum_{m=0}^{3n+2} P_{\Sigma_2}(3m+2) \prod_{\substack{t=0\\t\neq m}}^{3n+2} \frac{\ell - 3t - 2}{3m-3t}$, where\\[-10pt]
\begin{footnotesize}
\begin{align*}
&P_{\Sigma_2}(3m+2) = P_{\Sigma_1}(3m+2)  \\[-5pt]
&  +\sum_{i=1}^m 
\Bigg( P_{\Sigma_1}(3m-i+2)  + \Gamma \cdot l(3m-i+2) - \sum_{p=1}^{i+1} \Big( l(i+1-p)\cdot l(3m-2i+1+p)- \sum_{q=0}^{2i-2p+1} s_1(q,3m-i+2) \Big) \Bigg) \cdot l(i) \\[-5pt]
&+\sum_{i=1}^m
\Bigg( P_{\Sigma_1}(2i) + \Gamma \cdot l(2i) - \sum_{p=1}^i \Big( l(i-p) \cdot l(i+p)
- \sum_{q=0}^{2i-2p-1} s_1(q,2i) \Big) \Bigg) \cdot l(3m-2i+2) \\[-5pt]
&  +\sum_{i=1}^m \sum_{j=0}^{2i-1} \sum_{p=1}^{2i}  s_2(j, 2i-p, 3m+2)+  \sum_{i=1}^m \sum_{j=0}^{2i-1} \sum_{p=1}^j s_2(j-p, 2i+2p-1, 3m+2)\\[-5pt]
&   + \sum_{i=1}^m \sum_{j=0}^{2i-1} \sum_{k=0}^{3m-3i} \Bigg(\sum_{p=1}^{k+2i+2}  s_2(j, k+2i-p+2, 3m+2) + \sum_{p=1}^j s_2(j-p, k+2i+2p+1, 3m+2)\Bigg)~\quad~\\[-5pt]
& - \sum_{i=1}^m \sum_{j=0}^{2i-1} s_1(j, 3m-i+2) \cdot l(i)
-  \sum_{i=1}^m \sum_{j=0}^{2i-1} \sum_{k=0}^{3m-3i} s_1(j, 3m-i-k+1) \cdot  l(k+i+1).
\end{align*}
\end{footnotesize}
\end{enumerate}
\end{thm}

This theorem is a higher dimensional generalization of \cite[Theorem 4.4]{ENP2}, which deals with the case where $X$ is a curve. 
As in the theorem, assume that $k=1$ or $2$ and $L$ is sufficiently positive. Let $B^{k+1}=B^{k+1}(L):=\mathbb{P}(E_{k+1,L})$. Note that $H^i(X^{[k+1]}, S^{\ell} E_{k+1,L}) = H^i(B^{k+1}, \mathcal{O}_{B^{k+1}}(\ell))$ for $i \geq 0$ and $\ell \geq 1$. The complete linear system $\lvert \mathcal{O}_{B^{k+1}}(1) \rvert$ induces a birational morphism $\alpha_{k+1} \colon B^{k+1} \to \Sigma_k$, which is a resolution of singularities. We review basic facts on secant varieties in Section \ref{sec:prelim}. If one can determine the higher direct images $R^i \alpha_{k+1,*} \mathcal{O}_{B^{k+1}}$ explicitly, then Theorem \ref{thm:main1} follows from the Leray spectral sequence for $\alpha_{k+1}$ together with known results for the cohomology of secant varieties. This is precisely what we do in Section \ref{sec:R^ialpha}. We also briefly discuss in Remark \ref{rem:surfacecase} how to extend Theorem \ref{thm:main1} when $X$ is a surface and $k$ is arbitrary. The idea of computing the cohomology groups of symmetric powers of the tautological bundles via secant varieties was already explored in \cite[Section 4]{ENP2} in the case where $X$ is a curve. In that case, since  $\deg P_{\Sigma_k}(\ell)=\dim \Sigma_k = 2k+1$, the result is obtained by applying Lagrange interpolation after computing $P_{\Sigma_k}(\ell)$ for $-k \leq \ell \leq k+1$. These values can be derived rather easily from the descriptions of $R^i \alpha_{k+1,*} \mathcal{O}_{B^{k+1}}$. However, in the higher dimensional cases, we need a new strategy (we instead compute $P_{\Sigma_k}(km+k-1)$ for all $m \geq 0$), and the arguments become substantially more involved. See Section \ref{sec:hilbpoly} for the details. Finally, we note that the degrees of $\Sigma_1$ and $\Sigma_2$ were known before (\cite[Corollary 8.2.9]{joins.1999} and \cite[Theorem 1.2]{Choi}, respectively). Our result may be regarded as a far-reaching generalization of those work.

\section{Preliminaries}\label{sec:prelim}
\noindent We begin with collecting necessary facts on Hilbert schemes of points and secant varieties. We closely follow the notations in \cite{Choi.Lacini.Park.Sheridan.25}, and we refer to that paper and the references therein for further details. Throughout the paper, $X$ is a smooth projective complex variety of dimension $n$, and $L$ is a sufficiently positive line bundle on $X$ giving an embedding $X \subseteq \mathbb{P} H^0(X, L) = \mathbb{P}^r$. 

\subsection{Hilbert schemes of points}
The \emph{Hilbert scheme} $X^{[k]}:= \{ \xi \subseteq X \mid \dim \xi=0, \operatorname{length}(\xi)=k\}$ of $k$ points on $X$ is smooth if and only if $n \leq 2$ or $k \leq 3$ (see \cite{Cheah}). In this paper, we always assume that $k \leq 3$ so that $X^{[k]}$ is a smooth projective complex variety of dimension $kn$. For a zero-dimensional subscheme $\xi \subseteq X$ of length $k$, we denote by $[\xi]$ the point of $X^{[k]}$ corresponding to $\xi$. The universal family 
$$
\mathcal{Z}_{k}:=\{ (x, [\xi]) \in X \times X^{[k]} \mid x \in \xi\} \subseteq X \times X^{[k]}
$$
over $X^{[k]}$ has rational singularities (see \cite[\S 1.2.4]{Choi.Lacini.Park.Sheridan.25}). Let 
$\pr_1 \colon \mathcal{Z}_{k} \to X$ and $\pr_2 \colon \mathcal{Z}_{k} \to X^{[k]}$ be the projection maps, and $E_{k,L}:=\pr_{2,*} \pr_1^* L$ be the \emph{tautological bundle} associated to $L$, which is a rank $k$ vector bundle on $X^{[k]}$ since $\pr_2$ is a finite flat morphism of degree $k$. 
Note that $H^0(X^{[k]}, E_{k,L}) = H^0(X, L)$ and the fiber of $E_{k,L}$ over $[\xi] \in X^{[k]}$ can be naturally identified with $H^0(\xi, L|_{\xi})$. In particular, $E_{k,L}$ is globally generated whenever $L$ is $(k-1)$-very ample, i.e., $H^0(X, L) \to H^0(\xi, L|_{\xi})$ is surjective for every $[\xi] \in X^{[k]}$. 

Next, we consider the \emph{nested Hilbert scheme}
$$
X^{[k-1,k]}:=\{ ([\eta], [\xi]) \in X^{[k-1]} \times X^{[k]} \mid \eta \subseteq \xi \} \subseteq X^{[k-1]} \times X^{[k]}
$$
with projection maps $\tau \colon X^{[k-1, k]} \to X^{[k-1]}$ and $\rho \colon X^{[k-1,k]} \to X^{[k]}$. There is a residual morphism $\res \colon X^{[k-1,k]} \to X$ sending $([\eta], [\xi])$ to $(\mathscr{I}_{\xi}: \mathscr{I}_{\eta})$. When $n \leq 2$ or $k \leq 3$, it is also known that $X^{[k-1,k]}$ is a smooth projective complex variety of dimension $kn$ (see \cite{Cheah}). The morhpism $\rho \colon X^{[k-1,k]} \to X^{[k]}$ is generically finite and factors through the universal family $\mathcal{Z}_k$. The morphism $X^{[k-1,k]} \to \mathcal{Z}_k$ is a resolution of singularities. On the other hand, $X^{[k-1,k]}$ can be obtained by the blow-up $\bl \colon X^{[k-1, k]} \to X \times X^{[k-1]}$ along the universal family $\mathcal{Z}_{k-1}$ with exceptional divisor $F_{k-1}$. Then $\res = \pr_1 \circ \bl$ and $\tau = \pr_2 \circ \bl$, where $\pr_1 \colon X \times X^{[k-1]} \to X$ and $\pr_2 \colon X \times X^{[k-1]} \to X^{[k-1]}$ are projection maps. Notice that $\mathcal{Z}_2 = X^{[1,2]}$ is obtained by the blow-up of $X \times X$ along the diagonal $\Delta \subseteq X \times X$. We have a short exact sequence
$$
0 \longrightarrow \res^* L(-F_{k-1}) \longrightarrow \rho^* E_{k,L} \longrightarrow \tau^* E_{k-1,L} \longrightarrow 0.
$$

Now, we consider the \emph{Hilbert--Chow morphism} $h_k \colon X^{[k]} \to X^{(k)}$. The $k$-th \emph{symmetric product} $X^{(k)}$ is obtained by the quotient of the ordinary $k$-th product $X^k$ by the symmetric group action $\mathfrak{S}_k$ permuting the components. The line bundle $L^{\boxtimes k}$ on $X^k$ descends to a line bundle $S_{k,L}$ on $X^{(k)}$. We denote by $T_{k,L}:=h_k^* S_{k,L}$. Then there exists a divisor $\delta_k$ on $X^{[k]}$ such that $N_{k,L}:=\det E_{k,L} = T_{k,L}(-\delta_k)$. Notice that $2\delta_k$ is the exceptional divisor of $h_k$ when $n \geq 2$. We set $A_{k,L}:=T_{k,L}(-2\delta_k)$. When $k=1$, we have $\delta_1=0$ and $E_{1,L}=T_{1,L}=N_{1,L}=A_{1,L}=L$. We refer to \cite[\S 2.3.2]{Choi.Lacini.Park.Sheridan.25} for basic properties of these line bundles on $X^{[k]}$. Here we recall that $\rho^* T_{k,L} = \tau^* T_{k-1,L} \otimes \res^* L$ and $\rho^* \delta_k = \tau^* \delta_{k-1} + F_{k-1}$. We can also compute the canonical line bundles $\omega_{X^{[k]}} = T_{k, \omega_X}((n-2)\delta_k)$ and $\omega_{X^{[k-1,k]}} = (\tau^* \omega_{X^{[k-1]}} \otimes \res^*\omega_X)((n-1)F_{k-1})$ (see \cite[\S 1.1.7 and \S 1.2.3]{Choi.Lacini.Park.Sheridan.25}). Notice that $F_{k-1}$ is the ramification divisor of $\rho$ and $\omega_{X^{[k-1,k]}} = \rho^* \omega_{X^{[k]}}(F_{k-1})$.

Finally, we show the following lemma, which is certainly known to experts. 

\begin{lem}\label{lem:H^i(O_X^[2])}
We have
$$
H^i(X^{[2]}, \mathcal{O}_{X^{[2]}}) = \Bigg( \bigoplus_{\substack{p+q=i \\ p<q}} H^p(X, \mathcal{O}_X) \otimes H^q(X, \mathcal{O}_X) \Bigg) \oplus 
\begin{cases} 
S^2 H^{i/2}(X, \mathcal{O}_X) & \text{if $i$ is even and $i/2$ is even} \\
\wedge^2 H^{i/2}(X, \mathcal{O}_X) & \text{if $i$ is even and $i/2$ is odd} \\
0 & \text{$i$ is odd}.
\end{cases}
$$
\end{lem}

\begin{proof}
As $X^{(2)} = X^2/\mathfrak{S}_2$ has rational singularities and $h_2 \colon X^{[2]} \to X^{(2)}$ is a resolution of singularities, $H^i(X^{[2]}, \mathcal{O}_{X^{[2]}}) = H^i(X^{(2)}, \mathcal{O}_{X^{(2)}})$ for $i \geq 0$. Note that $\mathcal{O}_{X^{(2)}}=S_{2, \mathcal{O}_X}$ is obtained by descending $\mathcal{O}_{X^2}$. As in \cite[Lemma 5.25]{Choi.Lacini.Park.Sheridan.25}, $H^i(X^{(2)}, \mathcal{O}_{X^{(2)}})$ can be computed by analyzing $\mathfrak{S}_2$-action on the K\"{u}nneth decomposition of $H^i(X^2, \mathcal{O}_{X^2})=\bigoplus_{p+q=i} H^p(X, \mathcal{O}_X) \oplus H^q(X, \mathcal{O}_X)$.
\end{proof}

\subsection{Secant varieties}
We keep assuming $k=2$ or $3$ so that $X^{[k]}$ and $X^{[k-1,k]}$ are smooth. Most of the results still holds when $n \leq 2$ or $k \leq 3$. Let $B^k=B^k(L):=\mathbb{P}(E_{k+1,L})$ with canonical projection $\pi_k \colon B^k \to X^{[k]}$, and $H_k$ be the tautological divisor on $B^k$ so that $\mathcal{O}_{B^k}(H_k) = \mathcal{O}_{B^k}(1)$. As $E_{k,L}$ is globally generated, the tautological line bundle $\mathcal{O}_{B^k}(1)$ is base point free. Note that $H^0(B^k, \mathcal{O}_{B^k}(1)) = H^0(X^{[k]}, E_{k,L})=H^0(X, L)$. Thus the complete linear system $\lvert \mathcal{O}_{B^k}(1) \rvert$ gives a morhpism $\alpha_k \colon B^k \to \mathbb{P} H^0(X, L) = \mathbb{P}^r$ whose image is the \emph{$k$-secant variety}
$$
\sigma_k=\sigma_k(X, L):=\overline{\bigcup_{[\xi] \in X^{[k]}} \langle \xi \rangle} \subseteq \mathbb{P}^r,
$$
which is nothing but the $(k-1)$-th secant variety $\Sigma_{k-1}=\Sigma_{k-1}(X, L)$. Two different conventions for secant varieties are both widely used in the literature. From now on, following \cite{Choi.Lacini.Park.Sheridan.25}, we use the notation of $k$-secant variety $\sigma_k$ simply because it is more convenient for maintaining consistent indexing. Note that $\alpha_k \colon B^k \to \sigma_k$ is a birational morphism hence a resolution of singularities and $\sigma_k$ is a projective variety with $\dim \sigma_k = kn+k-1$. We have $\operatorname{Sing}(\sigma_k)=\sigma_{k-1}$ (see \cite[Corollary E]{Choi.Lacini.Park.Sheridan.25}). Here, recall some of the main results of \cite{Choi.Lacini.Park.Sheridan.25} (see also \cite{Chou.Song.18} and \cite{Ullery} for the case of $k=2$). 

\begin{thm}[{\cite[Theorems 5.3 and 5.4]{Choi.Lacini.Park.Sheridan.25}}]\label{thm:secantvariety}
For $k=2$ or $3$, we have the following:
\begin{enumerate}
    \item $\sigma_k$ has normal singularities, and $\sigma_k \subseteq \mathbb{P}^r$ is projectively normal.
    \item $H^i(\sigma_k, \mathscr{I}_{\sigma_{k-1}/\sigma_k}(\ell)) = H^i(\sigma_k, \mathcal{O}_{\sigma_k}(\ell))=0$ for $i>0$ and $\ell > 0$. 
    \item $\alpha_{k,*} \mathcal{O}_{B^k}(-Z_{k-1}^k)=\mathscr{I}_{\sigma_{k-1}/\sigma_k}$ and $R^i \alpha_{k,*} \mathcal{O}_{B^k}(-Z_{k-1}^k)=0$ for $i>0$.
\end{enumerate}
\end{thm}

\begin{rem}\label{rem:suffposcond}
In this paper, the assumption that the embedding line bundle $L$ is sufficiently positive is needed only when we apply Theorem \ref{thm:secantvariety}. Apart from this, most of the results introduced in this section only require that $L$ is $(2k-1)$-very ample. Recently, the authors of the present paper have proved that this theorem holds when $L=\omega_X \otimes A^m \otimes B$ with $m \geq 2n+2$ for $k=2$ and $m \geq 3n+3$ for $k=3$, where $A$ is very ample and $B$ is nef (see \cite{CP}). Therefore, the main results of this paper (Theorems \ref{thm:main1} and \ref{thm:main2}) remain valid under this assumption.
\end{rem}

Next, let $B^{k-1,k}=B^{k-1,k}(L):=\mathbb{P}(\tau^* E_{k-1,L})$ with canonical projection $\pi_{k-1,k} \colon B^{k-1,k} \to X^{[k]}$. We have a commutative diagram
$$
\xymatrix{
B^{k-1,k} \ar[r]^-{\widetilde{\tau}} \ar[d]_-{\pi_{k-1,k}} & B^{k-1} \ar[d]^-{\pi_{k-1}} \\
X^{[k-1,k]} \ar[r]_{\tau} & X^{[k-1]}.
}
$$
On the other hand, from the surjection $\rho^* E_{k,L} \to \tau^* E_{k-1,L}$, we get a morphism
$\alpha_{k-1,k} \colon B^{k-1,k} \to B^k$ whose image is denoted by $Z_{k-1}^k$. Note that $\alpha_{k-1,k} \colon B^{k-1,k} \to Z_{k-1}^k$ is a birational morphism hence a resolution of singularities and $Z_{k-1}^k:=\alpha_k^{-1}(\sigma_{k-1})$ is the exceptional divisor of $\alpha_k$. By \cite[Proposition 3.25]{Choi.Lacini.Park.Sheridan.25}, $\mathcal{O}_{B^3}(-Z_{k-1}^k) = \mathcal{O}_{B^k}(-k) \otimes \pi_k^* A_{k,L}$. There is a commutative diagram
$$
\xymatrix{
B^{k-1,k} \ar[r]^-{\alpha_{k-1,k}} \ar[d]_-{\pi_{k-1,k}} & Z_{k-1}^k \ar[d]^-{\pi_k} \\
X^{[k-1,k]} \ar[r]_{\rho} & X^{[k]}.
}
$$
Note that $\alpha_k \circ \alpha_{k-1,k} = \alpha_{k-1} \circ \widetilde{\tau}$ and $\alpha_k(\alpha_{k-1,k}(B^{k-1,k}))=\alpha_k(Z_{k-1}^k)=\sigma_{k-1}$. 
When $k=2$, we have $\mathcal{Z}_2 = X^{[1,2]} = B^{1,2} = Z_1^2$, which is the universal family over $X^{[2]}$ via $\pi_2$, and the restricted morphism $\alpha_2 \colon Z_1^2 \to X$ coincides with the residual morhpism $\res \colon X^{[1,2]} \to X$. For $k=3$, consider the divisor $\widetilde{\tau}^* Z_1^2$ on $B^{2,3}$. Let $Z_1^3:=\alpha_{2,3}(\widetilde{\tau}^*Z_1^2)$. Then $\alpha_3(Z_1^3)=X$ and $Z_1^3 = \alpha_3^{-1}(X)$. Note that $Z_1^3 = \mathcal{Z}_3$ is the universal family over $X^{[3]}$ via $\pi_3$ and the restricted morphism $\alpha_3 \colon Z_1^3 \to X$ coincides with the the projection map $\pr_1 \colon \mathcal{Z}_3 \to X$. Since $Z_2^3$ is a prime divisor on a smooth variety $B^3$, it is Gorenstein. We have $\omega_{B^{2,3}}(\widetilde{\tau}^* Z_1^2) = \alpha_{2,3}^* \omega_{Z_2^3}$. Observe that $\alpha_{2,3} \colon B^{2,3} \setminus \widetilde{\tau}^* Z_1^2 \to Z_2^3 \setminus Z_1^3$ is an isomorphism and $\alpha_{2,3} \colon \widetilde{\tau}^* Z_1^2 \to Z_1^3$ is a generically 2-to-1 map. It is easy to show that $\alpha_{2,3,*} \mathcal{O}_{B^{2,3}}(-\widetilde{\tau}^* Z_1^2) \rightarrow \mathcal{O}_{Z_2^3}$ is injective and the subscheme of $Z_2^3$ defined by $\alpha_{2,3,*} \mathcal{O}_{B^{2,3}}(-\widetilde{\tau}^* Z_1^2)$ is generically reduced. But $\alpha_{2,3,*} \mathcal{O}_{B^{2,3}}(-\widetilde{\tau}^* Z_1^2)$ and $\mathcal{O}_{Z_2^3}$ are Cohen--Macaulay, so $\alpha_{2,3,*} \mathcal{O}_{B^{2,3}}(-\widetilde{\tau}^* Z_1^2) = \mathscr{I}_{Z_1^3/Z_2^3}$ since the cosupport of $\alpha_{2,3,*} \mathcal{O}_{B^{2,3}}(-\widetilde{\tau}^* Z_1^2)$ is $Z_1^3$. 
Thus we have a short exact sequence
$$
0 \longrightarrow \alpha_{2,3,*} \mathcal{O}_{B^{2,3}}(-\widetilde{\tau}^* Z_1^2) \longrightarrow \mathcal{O}_{Z_2^3} \longrightarrow \mathcal{O}_{Z_1^3} \longrightarrow 0.
$$
Using Kawamata--Viehweg vanishing theorem, one can easily check $R^i \alpha_{2,3,*} \mathcal{O}_{B^{2,3}} = R^i \alpha_{2,3,*} \mathcal{O}_{B^{2,3}}(-\widetilde{\tau}^* Z_1^2) = 0$ for $i>0$. We refer to \cite[Sections 2 and 3]{Choi.Lacini.Park.Sheridan.25} for more details on $B^k$ and $B^{k-1,k}$.

\section{Cohomology groups of symmetric powers of tautological bundles}\label{sec:R^ialpha}

\noindent In this section, we prove Theorem \ref{thm:main1}. The main ingredient is the description of higher direct images $R^i \alpha_{k,*} \mathcal{O}_{B^k}$. The following proposition is a higher dimensional generalization of \cite[Theorem 4.2 (1)]{ENP2}, which deals with the case where $X$ is a curve. When $i=0$, we have $\alpha_{k,*} \mathcal{O}_{B^k} = \mathcal{O}_{\sigma_k}$ for $k=2$ and $3$ by Theorem \ref{thm:secantvariety} $(1)$.

\begin{prop}\label{prop:R^ialpha_*O_B}
Assume that $L$ is sufficiently positive. For $i \geq 1$, we have the following:
\begin{enumerate}
    \item $R^i \alpha_{2,*} \mathcal{O}_{B^2} = H^i(X, \mathcal{O}_X) \otimes \mathcal{O}_X$. In particular, $R^i \alpha_{2,*} \mathcal{O}_{B^2} = 0$ for $i \geq n+1$.
    \item $R^i \alpha_{3,*} \mathcal{O}_{B^3} = \big(H^i(X, \mathcal{O}_X) \otimes H^0(\Sigma_1, \mathcal{O}_{\Sigma_1}(\ell))\big) \oplus \big(H^i(X^{[2]}, \mathcal{O}_{X^{[2]}})/H^i(X, \mathcal{O}_X) \otimes H^0(X, \mathcal{O}_X(\ell))\big)$.
    In particular, $R^i \alpha_{3,*} \mathcal{O}_{B^3} = H^i(X^{[2]}, \mathcal{O}_{X^{[2]}}) \otimes \mathcal{O}_X$ for $n+1 \leq i \leq 2n$, and $R^i \alpha_{3,*} \mathcal{O}_{B^3} =0$ for $i \geq 2n+1$.
\end{enumerate}
\end{prop}

\begin{proof}
$(1)$ We have a short exact sequence
$$
0 \longrightarrow \mathcal{O}_{B^2}(-Z_1^2) \longrightarrow \mathcal{O}_{B^2} \longrightarrow \mathcal{O}_{Z_1^2} \longrightarrow 0.
$$
By Theorem \ref{thm:secantvariety} (3), $R^i \alpha_{2,*} \mathcal{O}_{B^2}(-Z_1^2) = 0$ for $i \geq 1$. Thus $R^i \alpha_{2, *} \mathcal{O}_{B^2} = R^i \alpha_{2, *} \mathcal{O}_{Z_1^2}$ for $i \geq 1$. Recall that $Z_1^2 = \mathcal{Z}_2$ is the universal family over $X^{[2]}$ and $\alpha_2|_{Z_1^2} \colon Z_1^2 \to X$ coincides with the projection map $\pr_1 \colon \mathcal{Z}_2 \to X$. Since $\mathcal{Z}_2$ is the blow-up of $X \times X$ along the diagonal $\Delta \subseteq X \times X$, it follows that $R^i \alpha_{2, *} \mathcal{O}_{B^2}  = R^i \pr_{1, *} \mathcal{O}_{\mathcal{Z}_2} = H^i(X, \mathcal{O}_X) \otimes \mathcal{O}_X$ for $i \geq 0$. 

\medskip

\noindent $(2)$ We have a short exact sequence
$$
0 \longrightarrow \mathcal{O}_{B^3}(-Z_2^3) \longrightarrow \mathcal{O}_{B^3} \longrightarrow \mathcal{O}_{Z_2^3} \longrightarrow 0.
$$
We fix $i \geq 1$. 
By Theorem \ref{thm:secantvariety} (3), $R^i \alpha_{3,*} \mathcal{O}_{B^3}(-Z_2^3) = 0$, so $R^i \alpha_{3, *} \mathcal{O}_{B^3} = R^i \alpha_{3, *} \mathcal{O}_{Z_2^3}$. 
Consider the commutative diagram 
$$
\xymatrix{
0 \ar[r] & \alpha_{2,3,*} \mathcal{O}_{B^{2,3}}(-\widetilde{\tau}^* Z_1^2) \ar[r] \ar@{=}[d] &  \mathcal{O}_{Z_2^3} \ar[r] \ar[d] & \mathcal{O}_{Z_1^3}  \ar[r] \ar[d] & 0 \\
0 \ar[r] & \alpha_{2,3,*} \mathcal{O}_{B^{2,3}}(-\widetilde{\tau}^* Z_1^2) \ar[r] & \alpha_{2,3,*} \mathcal{O}_{B^{2,3}} \ar[r] & \alpha_{2,3,*} \mathcal{O}_{\widetilde{\tau}^* Z_1^2} \ar[r] & 0,
}
$$
where two rows are short exact sequences. Notice that $R^j \alpha_{2,3,*} \mathcal{O}_{B^{2,3}}(-\widetilde{\tau}^*Z_{1^2}) = 0$ and $R^j \alpha_{2,3,*} \mathcal{O}_{B^{2,3}} = 0$ for all $j>0$. As $\alpha_3 \circ \alpha_{2,3} = \alpha_2 \circ \widetilde{\tau}$, we obtain a short exact sequence
$$
0 \longrightarrow R^i \alpha_{3,*} \alpha_{2,3,*} \mathcal{O}_{B^{2,3}}(-\widetilde{\tau}^* Z_1^2) \longrightarrow R^i \alpha_{3,*} \alpha_{2,3,*} \mathcal{O}_{B^{2,3}} \longrightarrow R^i \alpha_{3,*} \alpha_{2,3,*} \mathcal{O}_{\widetilde{\tau}^* Z_1^2} \longrightarrow 0,
$$
which can be identified with 
$$
0 \rightarrow H^i(X, \mathcal{O}_X) \otimes \mathscr{I}_{X/\sigma_2} \rightarrow \big(H^i(X, \mathcal{O}_X) \otimes \mathcal{O}_{\sigma_2} \big) \oplus \big( H^i(X^2, \mathcal{O}_{X^2})/H^i(X, \mathcal{O}_X) \otimes \mathcal{O}_X \big) \rightarrow H^i(X^2, \mathcal{O}_{X^2}) \otimes \mathcal{O}_X \rightarrow 0.
$$
Now, recall that $Z_1^3 = \mathcal{Z}_3$ is the universal family over $X^{[3]}$ and $\alpha_3|_{Z_1^3} \to X$ coincides with the projection map $\pr_1 \colon \mathcal{Z}_3 \to X$. 
There is a birational morphism $X^{[2,3]} \to \mathcal{Z}_3$ and $X^{[2,3]}$ is the blow-up of $X \times X^{[2]}$ along $\mathcal{Z}_2$. Thus $R^i \alpha_{3,*} \mathcal{O}_{Z_1^3} = H^i(X^{[2]}, \mathcal{O}_{X^{[2]}}) \otimes \mathcal{O}_X$, which is a direct summand of $H^i(X^2, \mathcal{O}_{X^2}) \otimes \mathcal{O}_X = R^i \alpha_{2,3,*} \mathcal{O}_{\widetilde{\tau}^* Z_1^2}$. Since the map $R^j \alpha_{3,*} \alpha_{2,3,*} \mathcal{O}_{B^{2,3}}(-\widetilde{\tau}^* Z_1^2) \longrightarrow R^j \alpha_{3,*} \mathcal{O}_{Z_2^3}$ is always injective for $j \geq 1$, we get a short exact sequence
$$
0 \longrightarrow R^i \alpha_{3,*} \alpha_{2,3,*} \mathcal{O}_{B^{2,3}}(-\widetilde{\tau}^* Z_1^2) \longrightarrow R^i \alpha_{3,*} \mathcal{O}_{Z_2^3} \longrightarrow R^i \alpha_{3,*} \mathcal{O}_{Z_1^3} \longrightarrow 0,
$$
and we also see that $R^i \alpha_{3,*} \mathcal{O}_{Z_2^3}$ is a direct summand of $R^i \alpha_{3,*} \alpha_{2,3,*} \mathcal{O}_{B^{2,3}}$. This shows the assertion $(2)$. 
\end{proof}

We are ready to give the proof of Theorem \ref{thm:main1}.

\begin{proof}[Proof of Theorem \ref{thm:main1}]
$(1)$ Since $B^2 = \mathbb{P}(E_{2,L})$, we have $H^i(X^{[2]}, S^{\ell} E_{2,L}) = H^i(B^2, \mathcal{O}_{B^2}(\ell))$ for $\ell \geq 1$. Note that $\mathcal{O}_{B^2}(\ell) = \alpha_2^* \mathcal{O}_{\sigma_2}(\ell)$. Recall from Theorem \ref{thm:secantvariety} (2) that $H^i(X, \mathcal{O}_{X}(\ell)) = 0$ and $H^i(\sigma_2, \mathcal{O}_{\sigma_2}(\ell))=0$ for $i>0$. As $\alpha_{2,*} \mathcal{O}_{B^2} = \mathcal{O}_{\sigma_2}$, we find $H^0(B^2, \mathcal{O}_{B^2}(\ell)) = H^0(\sigma_2, \mathcal{O}_{\sigma_2}(\ell))$. For $i \geq 1$, we have $R^i \alpha_{2,*} \mathcal{O}_{B^2} = H^i(X, \mathcal{O}_X) \otimes \mathcal{O}_X$ for $i \geq 1$ by Proposition \ref{prop:R^ialpha_*O_B}. Then the Leray spectral sequence for $\alpha_2$ shows that $H^i(B^2, \mathcal{O}_{B^2}(\ell)) = H^i(X, \mathcal{O}_X) \otimes H^0(X, \mathcal{O}_X(\ell))$.

\medskip

\noindent $(2)$ We have $B^3 = \mathbb{P}(E_{3,L})$, and we need to determine $H^i(B^3, \mathcal{O}_{B^3}(\ell))$ for $\ell \geq 1$. Note that $H^i(\sigma_3, \mathcal{O}_{\sigma_3}(\ell)) = 0$ for $i>0$ by Theorem \ref{thm:secantvariety} (2). We can conclude the assertion $(2)$ by applying Proposition \ref{prop:R^ialpha_*O_B} and the Leray spectral sequence for $\alpha_3$ in the same way as in $(1)$.
\end{proof}

\begin{rem}\label{rem:surfacecase}
Suppose that $X$ is a smooth projective complex surface and $L$ is sufficiently positive. In this case, Theorem \ref{thm:secantvariety} still holds for arbitrary $k$. Furthermore, $\sigma_k=\sigma_k(X, L)$ is Cohen--Macaulay if and only if $q(X)=h^1(X, \mathcal{O}_X)=0$ (in this case, $\sigma_k \subseteq \mathbb{P}^r$ is arithmetically Cohen--Macaulay so that $H^i(\sigma_k, \mathcal{O}_{\sigma_k}(\ell))=0$ for $1 \leq i \leq \dim \sigma_k-1 = 3k-2$ and $\ell \in \mathbb{Z}$), and $\sigma_k$ has rational singularities if and only if $q(X)=h^1(X, \mathcal{O}_X)=0$ and $p_g(X)=h^2(X, \mathcal{O}_X)=0$ (see \cite[Theorem A]{Choi.Lacini.Park.Sheridan.25}).
When $\sigma_k$ has rational singularities (i.e., $H^1(X, \mathcal{O}_X)=H^2(X, \mathcal{O}_X)=0$), we have $R^i \alpha_{k,*} \mathcal{O}_{B^k} = 0$ for $i >0$ so that $H^i(X^{[k]}, S^{\ell} E_{k,L}) = H^i(B^k, \mathcal{O}_{B^k}(\ell)) = 0$ for $i \geq 1$ and $\ell \geq 1$.
When $\sigma_k$ is Cohen--Macaulay (i.e., $H^1(X, \mathcal{O}_X=0$), we \emph{expect} that
$$
R^i \alpha_{k,*} \cO_{B^k} = \begin{cases} S^{i/2} H^2(X, \mathcal{O}_X)  \otimes \mathcal{O}_{\sigma_{k-i/2}} & \text{if $i$ is even and $0 \leq i \leq 2k-2$} \\
0 & \text{otherwise},\end{cases}
$$
which would imply that
$$
H^i(X^{[k]}, S^{\ell} E_{k,L}) = H^i(B^k, \mathcal{O}_{B^k}(\ell)) = \begin{cases} S^{i/2} H^2(X, \mathcal{O}_X)  \otimes H^0(\sigma_{k-i/2}, \mathcal{O}_{\sigma_{k-i/2}}(\ell)) & \text{if $i$ is even and $0 \leq i \leq 2k-2$} \\
0 & \text{otherwise} \end{cases}
$$
for $\ell \geq 1$. 
\end{rem}

\section{Hilbert polynomials of secant varieties}\label{sec:hilbpoly}

\noindent This section is devoted to the proof of Theorem \ref{thm:main2}. Throughout the section, we keep assuming that $L$ is a sufficiently ample line bundle on a smooth projective complex variety $X$ of dimension $n$. 

\subsection{Hilbert polynomial of $\sigma_2$}
We want to compute $P_{\sigma_2}(\ell) = \chi(\mathcal{O}_{\sigma_2}(\ell))$ for $\ell \in \mathbb{Z}$. We begin with the short exact sequence
\[
0 \longrightarrow \mathscr{I}_{X/\sigma_2} \longrightarrow \cO_{\sigma_2} \longrightarrow \cO_{X} \longrightarrow 0.
\]
We find
\[
\chi( \cO_{\sigma_2}(\ell)) = \chi(\cO_X(\ell)) + \chi (\mathscr{I}_{X/\sigma_2}(\ell)).
\] 
By Theorem \ref{thm:secantvariety} (3), we have 
\[
\chi (\mathscr{I}_{X/\sigma_2}(\ell)) = \chi(\mathcal{O}_{B^2}(\ell H_2 - Z_1^2)).
\]
Recall that $\mathcal{O}_{B^2}(-Z_1^2) = \mathcal{O}_{B^2}(-2) \otimes \pi_2^* A_{2,L}$. For an integer $i \geq 1$, we have a short exact sequence 
\[
0 \longrightarrow \cO_{B^2}(\ell H_2-(i+1)Z_1^2)
  \longrightarrow \cO_{B^2}(\ell H_2 -iZ_1^2)
  \longrightarrow 
  \cO_{Z_1^2}(\ell H_2 -iZ_1^2).
  \longrightarrow 0.
\]
Now, assume $\ell = 2m+1$ for some integer $m \geq 0$. As $\cO_{B^2}( \ell H_2 - (m+1)Z_1^2) = \cO_{B^3}(-1) \otimes \pi_2^* A_{2,L}^{m+1}$, we get $\chi( \cO_{B^2}( \ell H_2 - (m+1)Z_1^2) )=0$. Considering the above short exact sequences for $i=1, \ldots, m$, we see that
\[
\chi(\mathcal{O}_{B^2}(\ell H_2 - Z_1^2)) = \sum_{i=1}^m \chi (\cO_{Z_1^2}(\ell H_2 -iZ_1^2) ).
\]
Recall that $Z_1^2 = X^{[1,2]}$. Then we may write
\[
\cO_{Z_1^2}(\ell H_2 -iZ_1^2) =  \rho^* A_{2,L}^i \otimes \res^* L^{\ell - 2i} = (\tau^* L^i \otimes \res^* L^{\ell-i})(-2iF_1).
\]
By Lemma \ref{lem:chi(L^aL^b(-cF))} shown below, we have
\[
\chi\big( (\tau^* L^i \otimes \res^* L^{\ell-i})(-2iF_1) \big)  = \chi( L^{\ell-i}) \cdot \chi(L^i) - \sum_{j=0}^{2i-1} \chi(S^j \Omega_X \otimes L^{\ell}).
\]
Combining all together, we obtain
$$
\chi(\mathcal{O}_{\sigma_2}(2m+1)) = \chi(L^{2m+1}) + \sum_{i=1}^{m} \left(\chi( L^{2m-i+1}) \cdot \chi(L^i) - \sum_{j=0}^{2i-1} \chi(S^j \Omega_X \otimes L^{2m+1}) \right)
$$
for $m \geq 0$. Note that $P_{\sigma_2}(\ell) = \chi(\mathcal{O}_{\sigma_2}(\ell))$ is a polynomial in $\ell$ of degree $2n+1 = \dim \sigma_2$. By  Lagrange interpolation, 
$$
P_{\sigma_2}(\ell) = \sum_{m=0}^{2n+1} P_{\sigma_2}(2m+1) \prod_{\substack{t=0 \\ t \neq m}}^{2n+1} \frac{\ell - 2t-1}{2m - 2t}
$$
for any integer $\ell \in \mathbb{Z}$. To finish the proof of Theorem \ref{thm:main2} $(1)$, it only remain to show the following lemma.

\begin{lem}\label{lem:chi(L^aL^b(-cF))}
Let $F, G$ be vector bundles on $X$. 
For integers $a,b,c \geq 0$, we have
$$
\chi\big( (\tau^* (L^a \otimes F) \otimes \res^* (L^b \otimes G))(-cF_1) \big) = \chi(L^a \otimes F) \cdot \chi(L^b \otimes G) - \sum_{i=0}^{c-1} \chi(S^i \Omega_X \otimes F \otimes G  \otimes L^{a+b}).
$$
\end{lem}

\begin{proof}
Recall that $X^{[1,2]}$ is obtained by the blow-up $\bl \colon X^{[1,2]} \to X \times X$ along the diagonal $\Delta \subseteq X \times X$ with exceptional divisor $F_1$. We have
\[
(\tau^* (L^a \otimes F) \otimes \res^* (L^b \otimes G) )(-cF_1) = \bl^*( (L^b \otimes F) \boxtimes  (L^a \otimes G) )(-cF_1).
\]
For an integer $i \geq 0$, we have a short exact sequence
\[
0 \longrightarrow \cO_{X^{[1,2]}}(-(i+1)F_1)
  \longrightarrow \cO_{X^{[1,2]}}(-i F_1)
  \longrightarrow \cO_{F_1}(-iF_i)
  \longrightarrow 0.
\]
Notice that $\bl_* \cO_{X^{[1,2]}}(-i F_1) = \mathscr{I}_{\Delta/X^2}^i$ and $\bl_* \cO_{F_1}(-iF_1) = S^i N^{\vee}_{\Delta/X^2} = S^i \Omega_X$. Clearly, $R^j \bl_* \cO_{X^{[1,2]}}(-i F_1) = 0$ and $R^j \bl_* \cO_{F_1}(-iF_1) = 0$ for $j>0$ and $i \geq 0$. Considering the above short exact sequences for $i=0, \ldots, c-1$ after tensoring by $\bl^*( (L^b \otimes F) \boxtimes  (L^a \otimes G) )$, we see that
\[
\chi\big( \bl^*( (L^b \otimes G) \boxtimes  (L^a \otimes F) )(-cF_1) \big) = \chi( (L^b  \otimes G) \boxtimes (L^a \otimes F) ) - \sum_{i=0}^{c-1} \chi(S^i \Omega_X \otimes F \otimes G \otimes L^{a+b}).
\]
Finally, notice that $\chi( (L^b \otimes G) \boxtimes (L^a \otimes F) ) = \chi(L^b \otimes G) \cdot \chi(L^a \otimes F)$ by the K\"{u}nneth formula.
\end{proof}

\subsection{Hilbert polynomial of $\sigma_3$}
We want to compute $P_{\sigma_3}(\ell) = \chi(\mathcal{O}_{\sigma_3}(\ell))$ for $\ell \in \mathbf{Z}$. The basic strategy is the same as in the case of $\sigma_2$, but the computation becomes considerably more complicated. Rather than presenting every detailed step of the calculation, we focus on the essential ideas and key steps of the argument. The omitted part should be manageable for the reader. As before, we start by considering the short exact sequence
\[
0 \longrightarrow \cI_{\sigma_2/\sigma_3} \longrightarrow \cO_{\sigma_3} \longrightarrow \cO_{\sigma_2} \longrightarrow 0.
\]
We find
\[
\chi(\cO_{\sigma_3}(\ell))
= \chi (\cO_{\sigma_2}(\ell)) + \chi(\mathscr{I}_{\sigma_2/\sigma_3}(\ell))
= P_{\sigma_2}(\ell) + \chi(\mathscr{I}_{\sigma_2/\sigma_3}(\ell)).
\]
By Theorem \ref{thm:secantvariety} (3), we have
\[
\chi(\mathscr{I}_{\sigma_2/\sigma_3}(\ell)) = \chi(\cO_{B^3}(\ell H_3 - Z_2^3)). 
\]
Recall that $\cO_{B^3}(-Z_2^3) = \cO_{B^3}(-3) \otimes \pi_3^* A_{3,L}$. For an integer $i \geq 1$, we have a short exact sequence
\[
0 \longrightarrow \cO_{B^3}(\ell H_3-(i+1)Z_2^3)
  \longrightarrow \cO_{B^3}(\ell H_3 -i Z_2^3 )
  \longrightarrow \cO_{Z_2^3}(\ell H_3 - iZ_2^3)
  \longrightarrow 0.
\]
Now, assume $\ell = 3m+2$ for some integer $m \geq 0$. As $\cO_{B^3}(\ell H_3 - (m+1)Z_2^3) = \cO_{B^3}(-1) \otimes \pi_3^* A_{3,L}^{m+1}$, we get $\chi(\cO_{B^3}(\ell H_3 - (m+1)Z_2^3) ) = 0$. Considering the above short exact sequences for $i=1, \ldots, m$, we see that
\[
\chi(\cO_{B^3}(\ell H_3 - Z_2^3)) = \sum_{i=1}^m \chi(\cO_{Z_2^3}(\ell H_3 - iZ_2^3)).
\]
Next, from the short exact sequence
\[
0 \longrightarrow \alpha_{2,3,*}\cO_{B^{2,3}}(-\tilde{\tau}^*Z_1^2) \longrightarrow \cO_{Z_2^3} \longrightarrow \cO_{Z_1^3} \longrightarrow 0,
\]
we get
\[
\chi(\cO_{Z_2^3}(\ell H_3 - iZ_2^3)) = \chi( \alpha_{2,3,*}\cO_{B^{2,3}}(-\tilde{\tau}^*Z_1^2) \otimes \cO_{Z_2^3}(\ell H_3 - iZ_2^3)) + \chi(\cO_{Z_1^3}(\ell H_3 - iZ_2^3)).
\]
Note that
\begin{align*}
& \chi( \alpha_{2,3,*}\cO_{B^{2,3}}(-\tilde{\tau}^*Z_1^2) \otimes \cO_{Z_2^3}(\ell H_3 - iZ_2^3)))
= \chi \big( (\tau^* (S^{\ell - 3i-2} E_{2,L} \otimes A_{2,L}^{i+1}) \otimes \res^* L^i)(-2iF_2) \big);\\
& \chi(\cO_{Z_1^3}(\ell H_3 - iZ_2^3)) = \chi\big( (\tau^* A_{2,L}^i \otimes \res^* L^{\ell - 2i})(-2iF_2) \big).
\end{align*}
The Euler characteristics on the right-hand side are defined on $X^{[2,3]}$. For an integer $j \geq 0$, we have a short exact sequence
\[
0 \longrightarrow \cO_{X^{[2,3]}}(-(j+1)F_2) \longrightarrow \cO_{X^{[2,3]}}(-jF_2) \longrightarrow \cO_{F_2}(-jF_2) \longrightarrow 0.
\]
Recall that there is the blow-up $\bl \colon X^{[2,3]} \to X \times X^{[2]}$ along  $\mathcal{Z}_2=X^{[1,2]}$ with exceptional divisor $F_2$. Let $N:=N_{\mathcal{Z}_2/X \times X^{[2]}}$ be the normal bundle of $\mathcal{Z}_2$ in $X \times X^{[2]}$.
Considering the above short exact sequence for $j=0,\ldots, 2i-1$ after tesnsoring by relevant vector bundles on $X^{[2,3]}$, we see that
\begin{align*}
& \chi \big( (\tau^* (S^{\ell - 3i-2} E_{2,L} \otimes A_{2,L}^{i+1}) \otimes \res^* L^i)(-2iF_2) \big)
\\[-5pt]
& = \chi(S^{\ell - 3i-2} E_{2,L} \otimes A_{2,L}^{i+1}) \cdot \chi(L^i) - \sum_{j=0}^{2i-1} \chi(S^j N^{\vee} \otimes \rho^* (S^{\ell-3i-2} E_{2,L} \otimes A_{2,L}^{i+1}) \otimes \res^* L^i);\\
& \chi\big( (\tau^* A_{2,L}^i \otimes \res^* L^{\ell - 2i})(-2iF_2) \big)
= \chi(A_{2,L}^i) \cdot \chi(L^{\ell-2i}) - \sum_{j=0}^{2i-1}\chi(S^j N^{\vee} \otimes (\tau^* L^i \otimes \res^* L^{\ell-i})(-2iF_1)).
\end{align*}
The Euler characteristics on the right-end terms are defined on $X^{[1,2]}$. 
From the suitable symmetric powers of the short exact sequence 
\[
0 \longrightarrow  \tau^* L(-F_1) \longrightarrow \rho^* E_{2,L} \longrightarrow \res^*L \longrightarrow 0,
\]
we deduce that
\[
\chi(S^j N^{\vee} \otimes \rho^* (S^{\ell-3i-2} E_{2,L} \otimes A_{2,L}^{i+1}) \otimes \res^* L^i)
= \sum_{k=0}^{\ell-3i-2} \chi\big(S^j N^{\vee} \otimes (\tau^* L^{k + i+1} \otimes \res^* L^{\ell-i-k-1})(-(k+2i+2)F_1)\big).
\]
Putting together all the results obtained so far, we have
\begin{align*}
P_{\sigma_3}(3m+2) &= P_{\sigma_2}(3m+2) + \sum_{i=1}^m \Big(\chi(S^{3m - 3i} E_{2,L} \otimes A_{2,L}^{i+1}) \cdot \chi(L^i) + \chi(A_{2,L}^i) \cdot \chi(L^{3m-2i+2})\Big)\\[-5pt]
& ~ - \sum_{i=1}^m \sum_{j=0}^{2i-1} \chi\big(S^j N^{\vee} \otimes (\tau^* L^i \otimes \res^* L^{3m-i+2})(-2iF_1)\big)\\[-5pt]
& ~ -  \sum_{i=1}^m \sum_{j=0}^{2i-1} \sum_{k=0}^{3m-3i} \chi\big(S^j N^{\vee} \otimes (\tau^* L^{k + i+1} \otimes \res^* L^{3m-i-k+1})(-(k+2i+2)F_1)\big).
\end{align*}
for $m \geq 0$. 
In the end, it remains to compute $\chi(S^a E_{2,L} \otimes A_{2,L}^b)$ and $\chi(S^a N^{\vee} \otimes (\tau^* L^b \otimes \res^* L^c)(-dF_1))$. Finally, Theorem \ref{thm:main2} $(2)$ follows from the next two lemmas, which complete the remaining tasks, combined with Lagrange interpolation.

\begin{lem}
For integers $a, b \geq 0$, we have
\begin{align*}
& \chi(S^a E_{2,L} \otimes A_{2,L}^b) \\[-5pt] 
&= P_{\sigma_2}(a+2b) + \sum_{k=1}^n (-1)^k h^k(X, \mathcal{O}_X) \cdot \chi(L^{a+2b}) - \sum_{i=1}^b \Bigg( \chi(L^{b-i}) \cdot \chi(L^{a+b+i})
- \sum_{j=0}^{2(b-i)-1} \chi(S^j \Omega_X \otimes L^{a+2b}) \Bigg). 
\end{align*}
\end{lem}

\begin{proof}
Note that $\chi(S^a E_{2,L} \otimes A_{2,L}^b) = \chi(\cO_{B^2}(a) \otimes \pi_2^* A_{2,L}^b)$. As $\cO_{B^2}(-Z_1^2) = \cO_{B^2}(-2) \otimes \pi_2^* A_{2,L}$ and $Z_1^2= X^{[1,2]}$, we have a short exact sequence
$$
0 \longrightarrow \cO_{B^2}(a+2i-2) \otimes \pi_2^* A_{2,L}^{b-i+1} \longrightarrow \cO_{B_2}(a+2i)  \otimes \pi_2^* A_{2,L}^{b-i} \longrightarrow (\tau^* L^{b-i} \otimes \res^* L^{a+b+i} )(-2(b-i)F_1) \longrightarrow 0,
$$
for $i=1,\ldots, b$. We find
$$
\chi(\cO_{B^2}(a) \otimes \pi_2^* A_{2,L}^b) = \chi(\cO_{B^2}(a+2b)) - \sum_{i=1}^b \chi\big( (\tau^* L^{b-i} \otimes \res^* L^{a+b+i} )(-2(b-i)F_1)\big). 
$$
Now, by Theorem \ref{thm:main1} $(1)$,  
$$
\chi(\cO_{B^2}(a+2b)) = \chi(S^{a+2b} E_{2,L}) = P_{\sigma_2}(a+2b) + \sum_{k=1}^n (-1)^k h^k(X, \mathcal{O}_X) \cdot \chi(L^{a+2b}),
$$
and by Lemma \ref{lem:chi(L^aL^b(-cF))}, 
\[
\chi\big( (\tau^* L^{b-i} \otimes \res^* L^{a+b+i} )(-2(b-i)F_1)\big)
= \chi(L^{b-i}) \cdot \chi(L^{a+b+i}) - \sum_{j=0}^{2(b-i)-1} \chi(S^j \Omega_X \otimes L^{a+2b}). \qedhere 
\]
\end{proof}

\begin{lem}
For integers $a,b,c,d \geq 0$, we have
\begin{align*}
&\chi(S^a N^{\vee} \otimes (\tau^* L^b \otimes \res^* L^c)(-dF_1))\\[-5pt]
&= \chi(S^a \Omega_X \otimes L^c) \cdot \chi(L^b)  -  \sum_{i=1}^d \chi(S^a \Omega_X \otimes S^{d-i}\Omega_X \otimes L^{b+c}) -\sum_{j=1}^a \chi(S^{a-j}\Omega_X \otimes S^{d+2j-1} \Omega_X \otimes L^{b+c}) .
\end{align*}
\end{lem}

\begin{proof}
Recall that $N^{\vee}=N_{X^{[1,2]}/X \times X^{[2]}}^{\vee}$ is the conormal bundle of $X^{[1,2]}$ in $X \times X^{[2]}$.
We have two short exact sequences
$$
0 \longrightarrow N^{\vee}
\longrightarrow \res^*\Omega_X \longrightarrow \cO_{F_1}(-F_1) \longrightarrow 0
~\text{ and }~0 \lra \cO_{X^{[1,2]}}(-2F_1) \lra \cO_{X^{[1,2]}}(-F_1) \lra \cO_{F_1}(-F_1) \lra 0.
$$
The left short exact sequence can be derived from the fact that $F_1$ is the ramification divisor of $\rho \colon X^{[1,2]} \to X^{[2]}$ so that there is a short exact sequence
$$
0 \longrightarrow \rho^* \Omega_{X^{[2]}} \longrightarrow \Omega_{X^{[1,2]}} \longrightarrow \cO_{F_1}(-F_1) \longrightarrow 0.
$$
We have
$$
[\res^*\Omega_X] + [\cO_{X^{[1,2]}}(-2F_1)] = [N^{\vee}] + [\cO_{X^{[1,2]}}(-F_1)]
$$
in the Grothendieck group $K_0(X^{[1,2]})$. Note that $K_0(X^{[1,2]})$ is $\lambda$-ring (we refer to \cite[Chapter I]{FL RR algebra} for some details on $\lambda$-rings). Following \cite[Chapter I, \S1]{FL RR algebra}, for a locally free sheaf $F$ on $X^{[1,2]}$, we set
$$
 \lambda_t(F) := \sum_{i=0}^{\infty} [\wedge^iF]t^i \quad \text{and} \quad \sigma_t(F):= \lambda_{-t}(F)^{-1}=\sum_{i = 0}^{\infty} [S^i F] t^i.
$$
The last equality on the right-hand side holds by \cite[Corollary 2.4 in Chapter V]{FL RR algebra}. Then
$$
\sigma_t(\cO_{X^{[1,2]}}(-F_1)) = \frac{1}{1-[\cO_{X^{[1,2]}}(-F_1)]t}~~\text{ and }~~\sigma_t(\cO_{X^{[1,2]}}(-2F_1)) = \frac{1}{1-[\cO_{X^{[1,2]}}(-2F_1)]t}\;.
$$
From the following relation
$$
\sigma_t(N^{\vee}) \cdot \sigma_t(\cO_{X^{[1,2]}}(-F_1)) = \sigma_t(\res^*\Omega_X)\cdot \sigma_t(\cO_{X^{[1,2]}}(-2F_1)),
$$
we deduce that
$$
\sigma_t(N^{\vee}) = \sigma_t(\res^*\Omega_X) \cdot \frac{1- [ \cO_{X^{[1,2]}}(-F_1)]t}{1-[ \cO_{X^{[1,2]}}(-2F_1)]t}.
$$
Then we find
$$
[S^a N^{\vee}] = [S^a \res^*\Omega_X] + \sum_{i=1}^a \big( [\cO_{X^{[1,2]}}(-2iF_1)] - [\cO_{X^{[1,2]}}(-(2i-1)F_1)]  \big) \cdot [S^{a-i}\res^*\Omega_X].
$$
Considering the Grothendieck--Riemann--Roch theorem, we obtain
\begin{align*}
 &   \chi(S^a N^{\vee} \otimes (\tau^* L^b \otimes \res^* L^c)(-dF_1)) = \chi\big( (\tau^* L^b \otimes \res^*(S^a \Omega_X \otimes L^c))(-dF_1) \big) \\
 &   ~\quad~ + \sum_{i=1}^a \chi\big( (\tau^* L^b \otimes \res^*(S^{a-i} \Omega_X \otimes L^c))(-(d+2i)F_1) \big)  - \sum_{i=1}^a \chi\big( (\tau^* L^b \otimes \res^*(S^{a-i} \Omega_X \otimes L^c)(-(d+2i-1)F_1) \big).
\end{align*}
Applying Lemma \ref{lem:chi(L^aL^b(-cF))} and performing some straightforward calculations for a lot of cancellations, we can derive the claimed formula.
\end{proof}

\bibliographystyle{alpha}

\end{document}